\documentclass[11pt]{article} 
\usepackage[utf8]{inputenc}
\usepackage{amsfonts,amsmath,amsthm,amssymb}
\usepackage[justification=centering]{caption}
\usepackage{color}
\usepackage{fullpage}
\usepackage{enumitem}   
\usepackage{mathtools}
\usepackage{xcolor, xurl}
\usepackage{thmtools}
\usepackage{float}
\usepackage{hyperref}
\usepackage{cleveref}

\usepackage{thm-restate}
\makeatletter
\newcommand{\addresseshere}{%
  \enddoc@text\let\enddoc@text\relax
}
\makeatother

\newtheorem{theorem}{Theorem}[section]
\newtheorem{proposition}[theorem]{Proposition}
\newtheorem{lemma}[theorem]{Lemma}
\newtheorem{corollary}[theorem]{Corollary}
\newtheorem{remark}[theorem]{Remark}
\newtheorem{conjecture}[theorem]{Conjecture}
\theoremstyle{definition}
\newtheorem{example}[theorem]{Example}
\newtheorem{definition}[theorem]{Definition}
\newtheorem{openproblem}{Open Problem}

\usepackage{etoolbox}
\AtEndEnvironment{proof}{\setcounter{claim}{0}}
\newtheoremstyle{notation}
{1em}
{1em}
{}
{}
{\bfseries}
{.}
{.5em} 
{} 
\theoremstyle{notation}
\newtheorem{notation}[theorem]{Notation}

\newcommand*{\ST}{\mathcal{Q}}
\newcommand*{\des}{\mathrm{des}}

\newcommand*{\N}{\mathbb{N}}

\newcommand{\Sym}{\mathfrak{S}}

\newcommand{\set}[1]{\{#1 \}}

\newcommand{\bigset}[1]{\left\{ #1 \right\}}
\newcommand{\nun}{[n]_2}

\newcommand{\flatn}{\textnormal{flat}}

\usepackage[colorinlistoftodos]{todonotes}

\title{Flattened Stirling Permutations}

\usepackage{authblk}

\author[1]{Adam Buck}
\author[2]{Jennifer Elder}
\author[1]{Azia A. Figueroa}
\author[1]{Pamela E. Harris}
\author[1]{Kimberly J. Harry}
\author[3]{Anthony Simpson}

\affil[1]{Department of Mathematical Sciences, University of Wisconsin, Milwaukee, WI 53211\\
\textcolor{blue}{\href{mailto:adambuck@uwm.edu}{adambuck@uwm.edu}}, \textcolor{blue}{\textcolor{blue}{\href{mailto:aabarner@uwm.edu}{aabarner@uwm.edu}}, 
 \href{mailto:peharris@uwm.edu}{peharris@uwm.edu}},
 \textcolor{blue}{\href{mailto:kjharry@uwm.edu}{kjharry@uwm.edu}}
}

\affil[2]{Department of Computer Science, Mathematics and Physics, Missouri Western State University, St. Joseph, MO 64507\\
\textcolor{blue}{\href{mailto:flattenedparkingfunctions@outlook.com}{flattenedparkingfunctions@outlook.com}}}

\affil[3]{Greenwood, IN 46143\\
\textcolor{blue}{\href{mailto:anthony.lamont99@yahoo.com}{anthony.lamont99@yahoo.com}}}
\date{}

\begin{document}
\maketitle

\begin{abstract}
Recall that a Stirling permutation is a permutation on the multiset $\{1,1,2,2,\ldots,n,n\}$ such that any numbers appearing between repeated values of $i$ must be greater than $i$. 
We call a Stirling permutation ``flattened'' if the leading terms of maximal chains of ascents (called runs) are in weakly increasing order.
Our main result establishes a bijection between flattened Stirling permutations and type $B$ set partitions of $\{0,\pm1,\pm2,\ldots,\pm (n-1)\}$, which are known to be enumerated by the Dowling numbers, and we give an independent proof of this fact.
We also determine the maximal number of runs for any flattened Stirling permutation, and we enumerate flattened Stirling permutations with a small number of runs or with two runs of equal length. 
We conclude with some conjectures and generalizations worthy of future investigation.
\end{abstract}

\let\thefootnote\relax
\let\thefootnote\relax

\section{Introduction} 

Throughout, let $n \in \N \coloneqq \set{1,2,3,\ldots}$, $[n] \coloneqq \set{1, 2, \cdots, n}$, and $\Sym_n$ denote the set of permutations of the set $[n]$. 
If $\sigma \in\Sym_n$, then the \emph{runs} of $\sigma$ are the maximal contiguous increasing subwords of~$\sigma$.
If the sequence of leading terms of the runs appear in increasing order, 
then we say $\sigma$ is a \textit{flattened partition} of length $n$. 
For example, $\sigma=135246$ is a flattened partition of length six while $\sigma=613425$ is not a flattened partition, as the leading terms $6,1,2$ are not in increasing order.

Nabawanda, Rakotondrajao, and Bamunoba establish the following recursive formula for $f_{n,k}$, the number of flattened partitions of length $n$ with exactly $k$ runs~\cite[Theorem 1]{ONFRAB}: 
\begin{equation*}
            f_{n,k}= \sum_{m=1}^{n-2}\left({\binom{n-1}{m}} -1\right)f_{m,k-1},
    \end{equation*}
where $f_{n,1}=1$, for all $n\geq 1$, and $f_{n,k}=0$, for all $k\geq n \geq 2$.
Elder, Harris, Markman, Tahir, and Verga, generalized the work of Nabawanda et al.\  by giving recursive formulas for $\textbf{1}_r$-\textit{insertion flattened parking functions}. These are the flattened parking functions built from permutations of $[n]$, with $r$ additional ones inserted, and which have $k$ runs \cite[Theorems 29, 30 and 35]{flat_pf}.
Motivated by the work of Nabawanda et al.\ \cite{ONFRAB} and Elder et al.\ \cite{flat_pf}, we study \textit{flattened Stirling permutations}, a subset of Stirling permutations, which we define next.

We let $\nun \coloneqq\set{1,1,2,2,\ldots,n,n}$ denote the multiset with each element in $[n]$ appearing twice.
We recall that a Stirling permutation of order $n$ is a word $w=w_1w_2\ldots w_{2n}$ on the multiset $\nun$, in which every letter in $[n]$ appears exactly twice and any values between the repeated instances of $i$ must be greater than $i$.
We denote the set of Stirling permutations of order $n$ as $\ST_n$, and recall that $|\ST_n| = (2n-1)!!$  \cite[Theorem 2.1]{MR462961}.
For example, $112299388346677455\in \ST_{9}$, while $112293883946677545\notin \ST_{9}$ since the threes appears between the two instances of the nine.

Stirling permutations were introduced by Gessel and Stanley, 
and they give a combinatorial interpretation of the coefficients of the polynomial 
\[{Q_n(t)=\frac{(1-t)^{2n+1}}{t}\sum_{m\geq 0}S(m+n,m)t^m},\]
where $S(a,b)$ denotes the Stirling number of the second kind. 
They establish that \[Q_n(t)=\sum_{w\in\ST_n}t^{\des(w)},\] where $\des(w)$ counts the number of descents of $w$, which are indices $i$ such that $w_i>w_{i+1}$ \cite[Theorem 2.3]{MR462961}. 
Since then, many have studied statistics of  Stirling permutations, including enumerating their descents, ascents, ties, and mesa sets, as well as generalizing to quasi-Stirling permutations and studying descents and $\gamma$-positivity~\cite{DUH20182478,ELIZALDE2021105429,elizalde2022descents,MesaPaper,YAN2022112742}.

We now define our main object of study. Given $w \in\ST_n$,
we begin by defining the \emph{runs} of $w$ to be the maximal contiguous weakly increasing subwords of $w$.
If the sequence of leading terms of the runs appear in weakly increasing order, 
then we say $w$ is a \textit{flattened Stirling permutation} of order $n$.
We denote the set of flattened Stirling permutations of order $n$ by $\flatn(\ST_n)$, and we let $\flatn_k(\ST_n)$ be the subset of $\flatn(\ST_n)$, which have exactly $k$ runs. 
For example, 
$112299388346677455\in \ST_{9}$ is flattened since the leading terms of the runs appear in the order $1,3,3,4$; while 
 $123321445566778899 \in \ST_9$
 is not flattened, as the leading values of the runs are $1, 2, 1$, which are not in weakly increasing order. 
Table \ref{tab:data} provides the cardinality of the sets $|\ST_n|$, $|\flatn(\ST_n)|$, and $|\flatn_k(\ST_n)|$ for $1\leq n\leq 10$ and $1\leq k\leq 7$. The Sage code used to produce the data in this table can be found on \href{https://cocalc.com/share/public_paths/49995921347a438cdfdbf1bd39a740a325a909d2}{CoCalc}.

\begin{table}[h!]
\centering\resizebox{\textwidth}{!}{
\begin{tabular}{c | c | c | ccccccccc}
    $n$ & $|\ST_n|$ & $|\flatn(\ST_n)|$ &$|\flatn_{1}(\ST_n)|$ & $|\flatn_{2}(\ST_n)|$ & $|\flatn_{3}(\ST_n)|$ & $|\flatn_{4}(\ST_n)|$ & $|\flatn_{5}(\ST_n)|$ & $|\flatn_{6}(\ST_n)|$ & $|\flatn_{7}(\ST_n)|$ 
    \\
    \hline
    1 & 1& 1 & 1 &  &  & & \\
    2 & 3 & 2 & 1 & 1 & & & \\
    3 &15 &6 & 1 & 5 &  &  \\
    4 &105 & 24 &1 & 15 & 8 & &  \\
    5 &945 & 116 &1 & 37 & 70 & 8 & \\
    6 &10395 & 648 &1 & 83 & 374 & 190 &  \\
    7 &135135& 4088 & 1 & 177 & 1596 & 2034 & 280\\
    8 &2027025& 28640 & 1 & 367 & 6012 & 15260 & 6720 & 280 \\
    9 & 34459425& 219920 & 1 & 749 & 20994 & 93764 & 88732 & 15680\\
    10 &654729075 & 1832224 & 1 & 1515 & 69842 & 508538 & 866796 & 363132 & 22400\\
\end{tabular}
}
\caption{Counts for flattened Stirling permutations based on number of runs.}\label{tab:data}
\end{table}

Nabawanda et al.\   give a combinatorial bijection between
flattened partitions over $[n + 1]$ and set partitions of~$[n]$ \cite[Theorem 19]{ONFRAB}. 
Elder et al.\ also establish a bijection between $\textbf{1}_r$-insertion flattened parking functions with set partitions of $[n+r]$ with exactly $k$ blocks having two or more elements \cite[Theorem 37]{flat_pf}.
Given these results, it is natural to 
ask whether flattened Stirling permutations are in bijection with some type of set partitions. 
This is in fact true, and is the content of our main result, which we state below and illustrate in Figure~\ref{fig:B3} for $n=3$. 

\begin{theorem}
    \label{thm:main}
    For $n\geq 1$, the set $\flatn(\ST_{n+1})$ of flattened Stirling permutations of order $n+1$ is in bijection with the set of type $B$ set partitions on the interval of integers $[-n,n]$.
\end{theorem}

\begin{figure}[h]
\centering
\resizebox{\textwidth}{!}{
\begin{tikzpicture}

\draw[fill] (0,0) circle [radius=0.1];
\node at (0,-.5) {$0|1|2|3$};
\node at (0,-1) {\tiny{\textcolor{blue}{$11223344$}}};

\draw[fill] (0,0.7) circle [radius=0.1];
\node at (0,1.5) {$0|1|23$};
\node at (0,1) {\tiny{\textcolor{blue}{$11223443$}}};

\draw[fill] (-1.25,0.7) circle [radius=0.1];
\node at (1.3,1.5) {$0|\Bar{2}1|3$};
\node at (1.25,1) {\tiny{\textcolor{blue}{$11332244$}}};

\draw[fill] (1.25,.7) circle [radius=0.1];
\node at (-1.3,1.5) {$02|1|3$};
\node at (-1.3,1) {\tiny{\textcolor{blue}{$13312244$}}};

\draw[fill] (-2.5,.7) circle [radius=0.1];
\node at (-2.6,1.5) {$0|13|2$};
\node at (-2.55,1) {\tiny{\textcolor{blue}{$11244233$}}};

\draw[fill] (2.5,.7) circle [radius=0.1];
\node at (2.6,1.5) {$03|1|2$};
\node at (2.55,1) {\tiny{\textcolor{blue}{$14412233$}}};

\draw[fill] (-3.5,.7) circle [radius=0.1];
\node at (-3.8,1.5) {$0|12|3$};
\node at (-3.8,1.0) {\tiny{\textcolor{blue}{$11233244$}}};

\draw[fill] (3.5,.7) circle [radius=0.1];
\node at (3.8,1.5) {$0|\Bar{3}1|2$};
\node at (3.8,1) {\tiny{\textcolor{blue}{$11442233$}}};

\draw[fill] (4.8,.7) circle [radius=0.1];
\node at (5,1.5) {$0|1|\Bar{3}2$};
\node at (5,1) {\tiny{\textcolor{blue}{$11224433$}}};

\draw[fill] (-4.8,.7) circle [radius=0.1];
\node at (-5,1.5) {$01|2|3$};
\node at (-5,1) {\tiny{\textcolor{blue}{$12213344$}}};

\draw[fill] (-4.9,2) circle [radius=0.1]; 
\draw[fill] (-3.8,2) circle [radius=0.1];
\draw[fill] (-2.6,2) circle [radius=0.1]; 
\draw[fill] (-1.2,2) circle [radius=0.1];
\draw[fill] (0,2) circle [radius=0.1]; 
\draw[fill] (1.2,2) circle [radius=0.1];
\draw[fill] (2.6,2) circle [radius=0.1]; 
\draw[fill] (3.8,2) circle [radius=0.1];
\draw[fill] (4.9,2) circle [radius=0.1]; 
\draw[fill] (-8.9,4.5) circle [radius=0.1]; 
\draw[fill] (-7.5,4.5) circle [radius=0.1]; 
\draw[fill] (-6.0,4.5) circle [radius=0.1]; 
\draw[fill] (-4.5,4.5) circle [radius=0.1];
\draw[fill] (-3.0,4.5) circle [radius=0.1]; 
\draw[fill] (-1.5,4.5) circle [radius=0.1];
\draw[fill] (0,4.5) circle [radius=0.1]; 
\draw[fill] (1.5,4.5) circle [radius=0.1]; 
\draw[fill] (3.0,4.5) circle [radius=0.1]; 
\draw[fill] (4.5,4.5) circle [radius=0.1];
\draw[fill] (6.0,4.5) circle [radius=0.1]; 
\draw[fill] (7.5,4.5) circle [radius=0.1]; 
\draw[fill] (9,4.5) circle [radius=0.1]; 

\draw[thick](0,0)--(0,0.7);
\draw[thick](0,0)--(1.25,.7);
\draw[thick](0,0)--(-1.25,.7);
\draw[thick](0,0)--(-2.5,.7);
\draw[thick](0,0)--(2.5,.7);
\draw[thick](0,0)--(3.5,.7);
\draw[thick](0,0)--(-3.5,.7);
\draw[thick](0,0)--(4.8,.7);
\draw[thick](0,0)--(-4.8,.7);

\draw[very thin](-4.9,2)--(-9.0,4.5);
\draw[very thin](4.8,2)--(-9.0,4.5);

\draw[very thin](-4.9,2)--(-7.5,4.5);
\draw[very thin](-3.8,2)--(-7.5,4.5);
\draw[very thin](-1.2,2)--(-7.5,4.5);
\draw[very thin](1.2,2)--(-7.5,4.5);

\draw[very thin](-4.9,2)--(-6.0,4.5);
\draw[very thin](0,2)--(-6.0,4.5);

\draw[very thin](-3.8,2)--(-4.5,4.5);
\draw[very thin](-2.6,2)--(-4.5,4.5);
\draw[very thin](0,2)--(-4.5,4.5);

\draw[very thin](-2.6,2)--(-3.0,4.5);
\draw[very thin](-1.2,2)--(-3.0,4.5);

\draw[very thin](-1.2,2)--(-1.5,4.5);
\draw[very thin](0,2)--(-1.5,4.5);
\draw[very thin](2.6,2)--(-1.5,4.5);
\draw[very thin](4.9,2)--(-1.5,4.5);

\draw[very thin](-2.6,2)--(0,4.5);
\draw[very thin] (1.2,2)--(0,4.5);
\draw[very thin](4.9,2)--(0,4.5);

\draw[very thin](-1.2,2)--(1.5,4.5);
\draw[very thin](3.8,2)--(1.5,4.5);

\draw[very thin](0,2)--(3.0,4.5);
\draw[very thin](1.2,2)--(3.0,4.5);
\draw[very thin](3.8,2)--(3.0,4.5);

\draw[very thin](1.2,2)--(4.5,4.5);
\draw[very thin](2.6,2)--(4.5,4.5);

\draw[very thin](-4.9,2)--(6.0,4.5);
\draw[very thin](2.6,2)--(6.0,4.5);
\draw[very thin](3.8,2)--(6.0,4.5);

\draw[very thin](-3.8,2)--(7.5,4.5) ;
\draw[very thin](2.6,2)--(7.5,4.5);

\draw[very thin](-3.8,2)--(9,4.5);
\draw[very thin](3.8,2)--(9.0,4.5);
\draw[very thin](4.9,2)--(9.0,4.5);

\node at (-9.0,5.2) {$01|\Bar{3}2$};
\node at (-9.0,4.82) {\tiny{\textcolor{blue}{$12214433$}}};

\node at (-7.5,5.2) {$012|3$};
\node at (-7.5,4.82) {\tiny{\textcolor{blue}{$12233144$}}};

\node at (-6,5.2) {$01|23$};
\node at (-6,4.82) {\tiny{\textcolor{blue}{$12213443$}}};

\node at (-4.5,5.2) {$0|123$};
\node at (-4.5,4.82) {\tiny{\textcolor{blue}{$11233442$}}};

\node at (-3,5.2) {$02|13$};
\node at (-3,4.82) {\tiny{\textcolor{blue}{$13312442$}}};

\node at (-1.5,5.2) {$023|1$};
\node at (-1.5,4.82) {\tiny{\textcolor{blue}{$13344122$}}};

\node at (0,5.2) {$0|\Bar{2}13$};
\node at (0,4.82) {\tiny{\textcolor{blue}{$11332442$}}}; 

\node at (1.5,5.2) {$02|\Bar{3}1$};
\node at (1.5,4.82) {\tiny{\textcolor{blue}{$13314422$}}}; 

\node at (3,5.2) {$0|\Bar{2}\Bar{3}1$};
\node at (3,4.82) {\tiny{\textcolor{blue}{$11334422$}}}; 

\node at (4.5,5.2) {$03|\Bar{2}1$};
\node at (4.5,4.82) {\tiny{\textcolor{blue}{$14413322$}}};

\node at (6,5.2) {$013|2$};
\node at (6,4.82) {\tiny{\textcolor{blue}{$12244133$}}};

\node at (7.5,5.2) {$03|12$};
\node at (7.5,4.82) {\tiny{\textcolor{blue}{$14412332$}}};

\node at (9,5.2) {$0|\Bar{3}12$};
\node at (9,4.82) {\tiny{\textcolor{blue}{$11442332$}}};

\node at (0,7.6) {$0123$};
\node at (0,7.2) {\tiny{\textcolor{blue}{$12233441$}}};

\draw[fill] (0,6.8) circle [radius=0.1]; 

\draw[fill] (-9.0,5.7) circle [radius=0.1]; 
\draw[fill] (-7.5,5.7) circle [radius=0.1]; 
\draw[fill] (-6.0,5.7) circle [radius=0.1]; 
\draw[fill] (-4.5,5.7) circle [radius=0.1];
\draw[fill] (-3.0,5.7) circle [radius=0.1]; 
\draw[fill] (-1.5,5.7) circle [radius=0.1];
\draw[fill] (0,5.7) circle [radius=0.1]; 
\draw[fill] (1.5,5.7) circle [radius=0.1]; 
\draw[fill] (3.0,5.7) circle [radius=0.1]; 
\draw[fill] (4.5,5.7) circle [radius=0.1];
\draw[fill] (6.0,5.7) circle [radius=0.1]; 
\draw[fill] (7.5,5.7) circle [radius=0.1]; 
\draw[fill] (9,5.7) circle [radius=0.1]; 

\draw[very thin] (-9.0,5.7)--(0,6.8);
\draw[very thin] (-7.5,5.7)--(0,6.8);
\draw[very thin] (-6.0,5.7)--(0,6.8);
\draw[very thin] (-4.5,5.7)--(0,6.8);
\draw[very thin] (-3.0,5.7)--(0,6.8);
\draw[very thin] (-1.5,5.7)--(0,6.8);
\draw[very thin] (0,5.7)--(0,6.8);
\draw[very thin] (1.5,5.7)--(0,6.8);
\draw[very thin] (3.0,5.7)--(0,6.8);
\draw[very thin] (4.5,5.7)--(0,6.8);
\draw[very thin] (6.0,5.7)--(0,6.8);
\draw[very thin] (7.5,5.7)--(0,6.8);
\draw[very thin] (9,5.7)--(0,6.8);

\end{tikzpicture} 
}
    \caption{Hasse diagram for type $B$ set partitions of $[-3,3]$, illustrated in black, and the corresponding flattened Stirling permutations in $\ST_4$, listed in blue.}
    \label{fig:B3}
\end{figure}
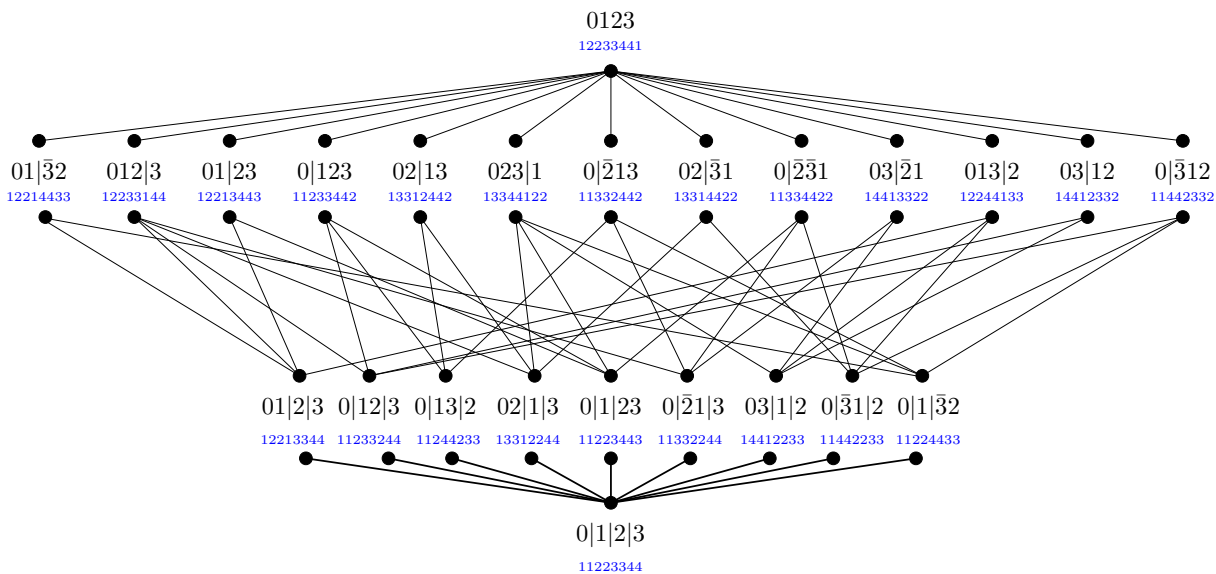

We recall that the
Dowling numbers, denoted $D_n$ (OEIS \href{https://oeis.org/A007405}{A007405}), enumerate type $B$ set partitions \cite[Statement 6]{adler}.
Thus, Theorem \ref{thm:main} readily implies that the set of flattened Stirling permutations of order ${n+1}$ are enumerated by the $n$th Dowling number (Corollary \ref{cor:dowling}).
Utilizing our bijection between type $B$ set partitions and flattened Stirling permutations, we give a proof for the following formula (Theorem \ref{them:compositions}):
\[D_n=|\flatn (\ST_{n+1})|=\displaystyle \sum_{i=0}^{n}\binom{n}{i}\sum_{k=0}^{n-i} 2^{n-i-k} S(n-i, k),\]
where $S(a,b)$ denote the Stirling numbers of the second kind with $S(a,0)=0$ for all $a\geq 0$.

Our paper is organized as follows.
In Section \ref{sec:background}, we provide the necessary background to make our approach precise,  including definitions and examples of type $B$ set partitions.
In Section~\ref{sec:main bijection}, we prove Theorem~\ref{thm:main}.
In Section \ref{sec:prelim results}, we give enumerative formulas for 
\begin{enumerate}[leftmargin=.25in]
    \item the number of runs in a flattened Stirling permutation of order $n$ (Lemma \ref{lemma:run count}),
    \item the maximal number of runs in a flattened Stirling permutation of order $n$ (Lemma \ref{lemma:max runs}),
    \item the number of flattened Stirling permutations of order $n$ with one or two runs (Proposition~\ref{prop:2RunRec}),
    \item the total number of flattened Stirling permutations of order $n$ (Theorem~\ref{them:compositions}).
\end{enumerate} 
In Section~\ref{sec:future}, we conclude by presenting a few conjectures and a generalization of our work worthy of future study.

\section{Background on type \texorpdfstring{$B$}{B} set  partitions}\label{sec:background}

We remark that the study of type $B$ set partitions began with the work of Montenegro, 
where he investigates the type $B$ analog of noncrossing set partitions from which he constructs the type $B$ associahedron \cite{typeBnoncrossing}.
Using this construction, Reiner \cite{REINER1997195} develops a definition for type $B$ set partitions. 
Adler expands Reiner's work and creates a notation for type $B$ set partitions as canonical words~\cite{adler}.
In this section, we provide the necessary definitions and adopt Adler's notation for type $B$ set partitions for our purposes.

\begin{definition}
A \textit{type $B$ set partition of} $[-n,n]$ is 
a collection of sets (also referred to as \textit{blocks}), 
 $\mathcal{B}=\{\beta_0,\beta_1,\beta_2,\ldots,\beta_k\}$, 
satisfying the following conditions:
\begin{enumerate}
\setlength\itemsep{0.2cm}
\item $\beta_i$ is a subset of $[-n,n]$ for all $0\leq i\leq k
$,\label{condition1}
\item $\beta_i\cap\beta_j=\emptyset$ whenever $i\neq j$, \label{condition2}
\item $\displaystyle\cup_{i=0}^k\,\beta_i=[-n,n]$,\label{condition3}
\item if $\beta\in\mathcal{B}$, then $-\beta\coloneqq\{-b:b\in\beta\}\in\mathcal{B}$, and \label{condition4}
        \item there is exactly one block $\beta\in\mathcal{B}$ satisfying $\beta = -\beta$. 
        \label{condition5}
    \end{enumerate}
\end{definition}

Note that Conditions \ref{condition1}-\ref{condition3} are the definition of a ``standard'' set partition of $[-n,n]$. We say the blocks $\beta$ and $-\beta$, described in Condition \ref{condition4}, are a \textit{block pair} of $\mathcal{B}$. 
The unique block described in Condition \ref{condition5}, satisfying $\beta=-\beta$, must contain $0$ and is called 
the \emph{zero-block of }$\mathcal{B}$.
We let $\Pi^B_n$ denote the collection of all type $B$ set partitions of $[-n,n]$, and we let $\Pi^B_{n,m}$ denote the collection of all type $B$ set partitions having exactly $m$ block pairs.

We  illustrate these definitions next.
\begin{example}\label{ex:2.1}
Consider the following collections of subsets of $[-4,4]$:\[\mathcal{B}_1 =\bigset{\set{-4,3}, \set{1,-2, 0,-1,2}, \set{-3,4} }\mbox{ and }
\mathcal{B}_2 =\left\{\set{3,-4, 2}, \set{1} ,\set{-1},  \set{-2,4,-3}, \set{0} \right\}\]
Notice that, $\mathcal{B}_1$ has zero-block $\set{1,-2, 0,-1,2}$ and a block pair consisting of  $\set{3,-4}$ and $\set{-3,4}$. 
Thus, $\mathcal{B}_1\in \Pi^B_{4,1}$.
Whereas, $\mathcal{B}_2$ has {zero-block $\set{0}$ and} two block pairs consisting of $\set{1}$, $\set{-1}$ and $\set{-2,-3,4}$, $\set{2,3,-4}$.
Thus, $\mathcal{B}_2 \in \Pi^B_{4, 2}$.
\end{example}

We adopt the notation of Adler, who encodes type $B$ set partitions through a canonical sequence of words, which we describe below~\cite[p.~2]{adler}.

\begin{notation}\label{theAdler}
Given $\mathcal{B} \in\Pi_{n,m}^B$, perform the following steps:
\begin{enumerate}
    \item Of the $m$ block pairs in $\mathcal{B}$, keep only the block with minimal positive element.\label{Adler1}
    \item Keep the $0$-block removing any negative elements and label the block $\beta_0$. \label{Adler2}
    \item Order the remaining blocks by minimum positive elements and label them $\beta_1,\ldots,\beta_m$. \label{Adler3}
    \item Order the full subcollection of block as $\beta_0,\beta_1,\ldots,\beta_m$. \label{Adler4}
    \item Within each block, order the elements by first placing the negatives in decreasing order, followed by the positive numbers in increasing order.\label{Adler5}
    \item For notational simplicity, denote negatives with bars, i.e, barred elements are just negative numbers. \label{Adler6}
    \item Encode the elements of $\beta_i$ as a word $\pi_i$ and denote the collection of blocks by 
    \[\pi=\pi_0|\pi_1|\pi_2|\cdots|\pi_m.\]\label{Adler7}
\end{enumerate}  
\end{notation}

We illustrate the use of Adler's notation (as described in Notation \ref{theAdler}) with the following example.

\begin{example}\label{ex:type_B_4} The type $B$ set partitions from Example \ref{ex:2.1} \[\mathcal{B}_1 =\bigset{\set{-4,3}, \set{1,-2, 0,-1,2}, \set{-3,4} }\mbox{ and }
\mathcal{B}_2 =\left\{\set{3,-4, 2}, \set{1} ,\set{-1},  \set{-2,4,-3}, \set{0} \right\}\] are encoded by 
$012\;|\;\bar{4}3    $ and $0\;|\;1\;|\;\bar{4}23$,
respectively.
Consider the following element of $\Pi_{10,4}^B$:
\vspace{-.02in}
\[\mathcal{B}=\{\{0\},\{1\}, \{-1\}, \{2, 7, -8\}, \{-2, -7, 8\}, \{3, 5, 6, -9, -10\}, \{-3, -5, -6, 9, 10\}, \{4\}, \{-4\}\}.\]
In this case, applying Notation \ref{theAdler} yields
$0\,|\,1\,|\,\bar{8}27\,|\,\bar{9}(\overline{10})356\,|\,4.$
\end{example}

Henceforth, all of our type $B$ set partitions adopt Notation \ref{theAdler}, and we utilize the canonical sequence of words to denote them. 
Therefore, we write $\pi\in\Pi_{n}^B$ to denote elements of the set for which we have implemented Notation~\ref{theAdler}.

\section{Main bijection}\label{sec:main bijection}

The main objective of this section is to prove Theorem \ref{thm:main}, establishing a bijection between flattened Stirling permutations of order $n+1$ and type $B$ set partitions of $[-n,n]$.
To begin, we adapt Notation~\ref{theAdler} for these purposes.

\begin{definition}\label{def:NP}
Let $\pi=\pi_0\,|\,\pi_1\,|\,\cdots\,|\,\pi_m$ be a type $B$ set partition in $\Pi_{n,m}^B$. 
For every $0\leq i\leq m$ further partition $\pi_i$ into 
$\pi_i=N_iP_i$, where:
\begin{itemize}
\item $N_i$ consists of the barred elements in $\pi_i$ and
\item $P_i$ consists of the positive values in $\pi_i$.
\end{itemize}
\end{definition}

Note that for any $\pi\in\Pi_{n,m}^B$,
 by Notation \ref{theAdler}, we know $\pi_0$ always consists of nonnegative values. Thus, it is always the case that $N_0=\emptyset$. Henceforth, we only consider $N_i$ with $1\leq i\leq m$.

\begin{example}\label{ex:N_iP_i}
Continuing Example \ref{ex:type_B_4}, let
$\pi=0\,|\,1\,|\,\bar{8}27\,|\,\bar{9}(\bar{10})356\,|\,4$. 
Then
\[\pi=\underbrace{0}_{P_0}\,|\,\underbrace{\emptyset}_{N_1}\underbrace{1}_{P_1}\,|\,\underbrace{\bar{8}}_{N_2}\underbrace{27}_{P_2}\,|\,\underbrace{\bar{9}(\bar{10})}_{N_3}\underbrace{356}_{P_3}\,|\,\underbrace{\emptyset}_{N_4}\underbrace{4}_{P_4}\,,\]
where we include an empty set symbol whenever a block did not have any barred elements.
\end{example}

We now provide the following technical result.

\begin{lemma}
Let $\pi=\pi_0\,|\,\pi_1\,|\,\cdots\,|\,\pi_m$ be a type $B$ set partition in $\Pi_{n,m}^B$. 
If $\pi_i=N_iP_i$ is as in Definition \ref{def:NP}, then for all $0\leq i\leq m$:
\begin{enumerate}
    \item $|N_i|\geq 0$,\label{c1}
    \item $|P_i|\geq 1$, and\label{c2}
    \item the minimal magnitude of the elements in $N_i$ is larger than the minimum element in $P_i$, i.e., $\min(-N_i)>\min(P_i)$, where $-N_i=\{-\ell:\ell\in N_i\}$.
    \label{c3}
\end{enumerate}
\end{lemma}
\begin{proof}
\noindent For \eqref{c1}: Any standard set partition would produce a type $B$ set partition in which $N_i=\emptyset$ for all $0\leq i \leq m$. In any other case, $|N_i|\geq 0$.

\noindent For \eqref{c2}: By definition, we always select a block pair by minimum of the positive elements. Thus there is at least one element in the corresponding $P_i$. Implying that $|P_i|\geq 1$. 

\noindent For \eqref{c3}: 
If $N_i=\emptyset$, then by definition, the minimal magnitude of the elements in $N_i$ is infinity. This is larger than the minimum element in $P_i$.
If $N_i\neq \emptyset$ and  $\min(-N_i) < \min(P_i)$, then this would contradict that $\pi_i=N_iP_i$ was the block pair which had the smallest nonnegative positive value, as $-\pi$ would have been selected in Step~\ref{Adler1} of Notation \ref{theAdler}.
\end{proof}

We now establish the following results, which will be used in our main bijection.

\begin{lemma}\label{lem:function h}
Let $P([-n,n])$ denote the power set of the set of integers in the interval $[-n,n]$.
For any $i\in\mathbb{Z}$, let $|i|$ denote the absolute value of $i$.
Then the function $h:P([-n,n])\to P([n+1])$ defined 
by 
\[h(S)=\begin{cases}
\{|i|+1\,:\, i\in S\}&\mbox{if $S\neq\emptyset$}\\
\emptyset&\mbox{if $S=\emptyset$}
\end{cases}
\]
is well-defined.
\end{lemma}

\begin{proof}
Let $S\in P([-n,n])$ and note that $0\leq |S|\leq 2n+1$.   
If $S=\emptyset$, then $h(S)=\emptyset$, which is well-defined.
Suppose $1\leq |S|\leq k$. 
If $-i,i\in S$, then $|-i|+1=i+1$ appears exactly once in $h(S)$, as $h(S)$ is a set and not a multiset.
Then $h(S)=\{|i|+1\,:\, i\in S\}$ has size at most $k$ and whenever $j\in h(S)$ it satisfies $1\leq j\leq n+1$. Hence $h(S)\in P([n+1])$. Implying that $h$ is well-defined.
\end{proof}

Let $\mathcal{W}([n+1])$ denote the set of all words of length $2k$, with $0\leq k\leq n+1$, whose letters come from the alphabet $[n+1]$ and each letter that appears in a word appears exactly twice. 
For example, if $n=1$, then $\mathcal{W}([2])=\{\emptyset,11,22,1122,1212,1221,2112,2121,2211\}$, where we let $\emptyset$ denote the empty word, which has length zero.
Note that $\ST_{n+1}\subseteq \mathcal{W}([n+1])$ for all $n\geq 0$.

\begin{lemma}\label{lem:functions f and g}
Let $P([n+1])$ denote the power set of $[n+1]$.
Then the functions $f:P([n+1])\to\mathcal{W}([n+1])$ defined 
by 
\[f(S)=\begin{cases}
s_1~s_1~s_2~s_2~s_3~s_3~\cdots ~s_k~s_k& \mbox{if $S=\{s_1<s_2<s_3<\cdots<s_k\}$}\\
\emptyset&\mbox{if $S=\emptyset$}\\\end{cases}
\] and 
$g:P([n+1])\to \mathcal{W}([n+1])$ defined by 
\[g(S)=\begin{cases}
s_1~s_2~s_2~s_3~s_3~\cdots ~s_k~s_k~s_1& \mbox{if $S=\{s_1<s_2<\cdots<s_k\}$}\\
\emptyset&\mbox{if $S=\emptyset$}\\\end{cases}\] are well-defined.
\end{lemma}
\begin{proof}
If $S=\emptyset$, then  $f(\emptyset)$ and $g(\emptyset)$ return the unique output $\emptyset$, which is an element of $\mathcal{W}([n+1])$.
Given a nonempty subset $S$ of $[n+1]$, by the well-ordered property of the natural numbers, we can uniquely order    the elements of $S$ as $s_1<s_2<s_3<\cdots<s_k$, where $k=|S|$. 
Then $f(S)=s_1~s_1~s_2~s_2~\cdots ~s_k~s_k$ and $g(S)=s_1~s_2~s_2~\cdots ~s_k~s_k~s_1$ are unique outputs of these functions. These are
words with letters in $[n+1]$, each appearing twice. Thus, the words have length $2k$ and are elements of $\mathcal{W}([n+1])$.
Thus, the functions $f$ and $g$ are well-defined.
\end{proof}

\begin{proposition}\label{forwardmap}
    Let $h$ be as defined in Lemma \ref{lem:function h}, 
    and $f$, $g$ be as defined in Lemma \ref{lem:functions f and g}.
    Let $\pi=\pi_0\,|\,\pi_1\,|\,\cdots\,|\,\pi_m\in\Pi_{n}^B$ and, for each $0\leq i\leq m$, let 
$\pi_i=N_iP_i$ be as in Definition \ref{def:NP}. 
    The function $\varphi:\Pi_{n}^B\to\flatn (\ST_{n+1})$ defined by
    \begin{equation}\label{eq:forwardmap}
        \varphi(\pi)=g(h(P_0))f(h(N_1))g(h(P_1))\cdots f(h(N_m))g(h(P_m))
    \end{equation}
    is well-defined. 
\end{proposition}

\begin{proof}
By Lemmas \ref{lem:function h} and \ref{lem:functions f and g} and the fact that compositions of well-defined functions are well-defined, we know the output of $\varphi(\pi)$ is unique and consists of the concatenation of well-defined functions. 
Moreover, by the definition of $\pi=\pi_1\,|\,\pi_2\,|\,\cdots\,|\,\pi_m$, 
we know that the numbers $[0,n]$ appear exactly once in $\pi$, some of which might have bars.
Furthermore, for each $i\in[m]$, we have that 
$\pi_i=N_iP_i$ where
$N_i$ consists of the barred/negative elements in $\pi_i$ and
$P_i$ consists of the positive values in $\pi_i$.
Note that applying the function $h$ to $N_i$ (for all $1\leq i\leq m$) and to $P_i$ (for all $0\leq i\leq m$), returns all of the numbers in $[n+1]$ and each number appears exactly once.
Applying $g$ and $f$ as described in Equation \eqref{eq:forwardmap} ensures that the output $\varphi(\pi)$ is a word of length $2(n+1)$ in which the numbers in $[n+1]$ appear exactly twice. 

It suffices to show that $\varphi(\pi)$ is a Stirling permutation of order $n+1$, which satisfies the flattened condition\footnote{For the interested reader, we provide Example~\ref{ex:running2} illustrating this part of the proof.}. We do this in that order.\\

\noindent\textbf{Step 1.} Showing $\varphi(\pi)$ is in $\ST_{n+1}$.\\

To show that $\varphi(\pi)\in \ST_{n+1}$, fix $i\in[n+1]$ to be arbitrary. 
There are two cases to consider:
\begin{itemize}
    \item The first case is if 
$i$ appears in an outcome of the form $f(h(N_j))$ for some $1\leq j\leq m$, then $i\,i$ appears in $\varphi(\pi)$ with no values in between the repeated instances of $i$. 
Thus, satisfying the needed condition for a Stirling permutation. \\
    \item The second case is if $i$ appears in an outcome of the form $g(h(P_j))$ for some $0\leq j\leq m$.
In which case, 
if $i=\min(P_j)$, then we have that $i$ appears as the first and last letters in the subword $g(h(P_j))$. All values between the repeated instances of $i$ (if any) are larger than~$i$. Thus, satisfying the needed condition for a Stirling permutation. If $i\neq\min(P_i)$, then we have that $i\,i$ appears in $\varphi(\pi)$ with no values in between the repeated instance of $i$. We also note that $i$ is larger than the $\min(P_i)$. Thus, $\varphi(\pi)$ satisfies the needed condition of a Stirling permutation.
\end{itemize}
\smallskip

\noindent \textbf{Step 2.} Showing that $\varphi(\pi)$ is flattened. \\

Recall from Notation \ref{theAdler}, that we ordered the blocks of $\pi$ based on the minimum positive values and placed the zero-block at the start.
Moreover, each block was then ordered by placing any negative values within a block at the start of the block in decreasing order followed by the positive values in increasing order (Step~\ref{Adler5} in Notation~\ref{theAdler}).

We begin by claiming that no run can be started within a subword $f(h(N_i))$, for any $1\leq i\leq m$. 
To establish this, note that the subword $g(h(P_{i-1}))$ will always end with the minimum of $P_{i-1}$. 
Since every element in $h(N_i)$ is larger than the minimum of $h(P_i)$, and $\min(h(P_{i-1}))<\min(h(P_i))$. 
Then $f(h(N_i))$ (which is written in weakly increasing order) will never start a run.

Fix $1\leq i\leq m$. We now consider the following exhaustive cases:
{\setlength{\leftmargini}{.3in}
\begin{enumerate}
    \item If $|P_i|=1$ and $N_i=\emptyset$, then $g(h(P_i))$ does not begin a new run because the $\min(P_{i-1})<\min(P_i)$.
    \item If $|P_i|=1$ and $N_i\neq \emptyset$, then $f(h(N_i))g(h(P_i))$ is equal to
    \[
    (|n_1|+1) \, (|n_1|+1)\, (|n_2|+1) \,(|n_2|+1)\, \cdots \, (|n_{|N_i|}|+1)\, (|n_{|N_i|}|+1)\, (p_1+1)\, (p_1+1),
     \]
    where $N_i=\{n_1>n_2>\cdots> n_{|N_i|}\}$, $P_i=\{p_1\}$, and $p_i<|n_{|N_i|}|+1$. Thus, the first instance of  $p_1+1=\min(P_i)+1$ begins a new run, and the second instance continues this same run. 
    \item If $|P_i|>1$ and $N_i= \emptyset$, then $f(h(N_i))g(h(P_i))=g(h(P_i))$ is equal to 
    \[(p_1+1)\, (p_2+1)\,(p_2+1)\,\cdots \,(p_{|P_i|}+1)\,(p_{|P_i|}+1)\, (p_1+1),\]
    where $P_i=\{p_1<p_2<\cdots<p_{|P_i|}\}$. Thus, 
    the first occurrence of $p_1+1=\min(P_i)+1>\min(P_{i-1})+1$ and hence continues the previous run (if any) and the second instance of $p_1+1=\min(P_i)+1$ begins a new run.
    \item If $|P_i|>1$ and $N_i \neq \emptyset$, then $f(h(N_i))g(h(P_i))$ is equal to
    \[
    \begin{array}{c}
    \resizebox{1 \textwidth}{!}{ $    
    (|n_1|+1) \, (|n_1|+1)\,  \cdots \, (|n_{|N_i|}|+1)\, (|n_{|N_i|}|+1)\,
    (p_1+1)\, (p_2+1)\,(p_2+1)\,\cdots \,(p_{|P_i|}+1)\,(p_{|P_i|}+1)\, (p_1+1),
    $     }    \end{array}
    \]
    where $P_i=\{p_1<p_2<\cdots<p_{|P_i|}\}$
    and $N_i=\{n_1>n_2>\cdots>n_{|N_i|}\}$.
    Thus, 
    each instance of $p_1+1=\min(P_i)+1$ begins a new run since $|n_{|N_i|}|+1>p_1+1$ and $p_{|P_i|}+1>p_1+1$.
\end{enumerate}
}

\bigskip

For the case where $i=0$, note that $N_0$ is empty and 
\[g(h(P_0))=\begin{cases}
    11&\mbox{if $P_0=\{1\}$}\\
    1\,p_2\,p_2\,\cdots p_{k}\, p_{k}\, 1&\mbox{if $P_0=\{1<p_2<p_3<\cdots<p_{k}\}$,  with $k\geq 2$.}\\
\end{cases}\] 
This ensures that $\varphi(\pi)$ begins with $11$ or that the first two runs begin with  the value $1$. 

Note that by the Step \ref{Adler3} in Notation~\ref{theAdler}, we know that $\min(P_0)< \min(P_1)< \cdots < \min(P_m)$.
Hence, ensuring that $(\min(P_0)+1)< (\min(P_1)+1)< \cdots < (\min(P_m)+1)$. As the only
locations at which runs may begin are at values equal to $\min(P_i)+1$ (for some $i$), we have that $\varphi(\pi)$ satisfies the flattened condition.
Thus, 
$\varphi(\pi)$ is a flattened Stirling permutation.\\

Together \textbf{Steps 1} and \textbf{Step 2} establishes that $\varphi$ is well-defined.
\end{proof}

We illustrate the proof technique for Proposition \ref{forwardmap} below.

\begin{example}\label{ex:running2}
Continuing Example \ref{ex:N_iP_i}, if
\[\pi=\underbrace{0}_{P_0}\,|\,\underbrace{\emptyset}_{N_1}\underbrace{1}_{P_1}\,|\,\underbrace{\bar{8}}_{N_2}\underbrace{27}_{P_2}\,|\,\underbrace{\bar{9}(\bar{10})}_{N_3}\underbrace{356}_{P_3}\,|\,\underbrace{\emptyset}_{N_4}\underbrace{4}_{P_4}\]
then
\[
\varphi(\pi) = \underbrace{11}_{g(h(P_0))}\underbrace{22}_{g(h(P_1))}\underbrace{99}_{f(h(N_2))}\underbrace{3883}_{g(h(P_2))} \underbrace{(10)(10)(11)(11)}_{f(h(N_3))} \underbrace{466774}_{g(h(P_3))} \underbrace{55}_{g(h(P_4))}.
\]
Note that $\varphi(\pi)$ is an element of $\flatn(\ST_{11})$ and has five runs.
\end{example}

We now prove that the function defined in Proposition~\ref{forwardmap} is a bijection.

\begin{remark}
    Before proving the main result, we remark that an alternate proof may be given using the language of ``left-to-right minima.'' However, our proof is elementary, utilizes the least amount of technical language, and the use of left-to-right minima does not necessarily shorten the arguments provided.
\end{remark}

\begin{theorem}\label{thm:bijection}
    The function $\varphi:\Pi_{n}^B\to\flatn(\ST_{n+1})$ defined in Proposition \ref{forwardmap} is a bijection. 
\end{theorem}

\begin{proof}
For injectivity, assume that $\pi=\pi_0\,|\,\pi_1\,|\,\pi_2\,|\,\cdots|\pi_m$ and $\tau=\tau_0\,|\,\tau_1\,|\,\tau_2\,|\,\cdots\,|\,\tau_k$ are distinct elements of $\Pi_{n}^B$. Further, suppose that the first index at which $\pi$ and $\tau$ differ is $i$. Namely, $\pi_j=\tau_j$ for all $0\leq j<i $ and $\pi_i\neq\tau_i$. 
Let $\pi_i=N_iP_i$ and $\tau_i=N_i'P_i'$ be as in Definition \ref{def:NP}.
If $N_i\neq N_i'$, then 
$h(N_i)\neq h(N_i')$. Hence 
$f(h(N_i))\neq f(h(N_i'))$, and so $\varphi(\pi)\neq\varphi(\tau)$.
If $N_i=N_i'$, then it must be that $P_i\neq P_i'$, hence 
$h(P_i)\neq h(P_i')$. Hence 
$g(h(P_i))\neq g(h(P_i'))$, and so $\varphi(\pi)\neq\varphi(\tau)$.
Thus $\varphi$ is injective.

For surjectivity,\footnote{For the interested reader, we provide Example~\ref{ex:thmex} illustrating this part of the proof.} 
    let $w = w_1\,w_2\,\ldots w_{2n}\in \flatn(\ST_n)$. We will construct the subwords $N_i^*$ and $P_i^*$ as follows.
   
    \begin{enumerate}
           \item Start at the right end of $w$, and consider the value $w_{2n}$.\label{enum:firstProcess} 
           \item Find the smallest index $1\leq j\leq 2n-1$ for which $w_j=w_{2n}$. Note $w_j$ is the left most occurrence of the value $w_{2n}$. Label the subword beginning at $w_j$ and ending at $w_{2n}$, which we denote with $w_j\cdots w_{2n}$, as $P_{1}^*$.
        \item Look at $w_{j-1}$. 
        \begin{enumerate}
            \item If $w_{j-1}<w_j$, then $N_1^* = \emptyset$
            \item If $w_{j-1}>w_j$, then find the longest contiguous subword $w_{j-m} \cdots w_{j-1}$ such that $w_{x}>w_j$ for all $j-m\leq x \leq j-1$. Label the subword beginning at $w_{j-m}$ and ending at $w_{j-1}$, which we denote with $w_{j-m} \cdots w_{j-1}$,  as $N_1^*$. 
        \end{enumerate}
    \item Repeat this process to construct $P_2^*$ and $N_2^*$, starting at the next available letter, $w_{j-m-1}$, which is not included in $N_1^*P_1^*$. Continue constructing $P_i^*$ and $N_i^*$ in this way, until you have labeled all of $w$ in this manner and included $w_1$ in some $P_k^*$ for some $k\geq 1$.\label{enum:firstprocessend}
    \end{enumerate}
    Note that since $w\in \flatn(\ST_n)$, we will end after $k\geq 1$ iterations of the steps above. We end with $P_k^*$ which is a subword that begins and ends with the value $1$. We now have a partitioning of $w$ into the subwords
    \[
    w=P_{k}^*N_{k-1}^*P_{k-1}^*\cdots N_{1}^*P_{1}^*.
    \]

Create a type $B$ set partition from these subwords as follows:

\begin{enumerate}[start=5]
    \item Start at $P_k^*$. Delete the right most occurrence of each letter. Replace each remaining letter, $w_i$, with $(w_i-1)$. Draw a divider to the right of these letters. Label this subword as $\pi_0 = P_0$.\label{enum:secondprocess} 
    \item If $k>1$, then consider $N_{k-1}^*P_{k-1}^*$. 
    \begin{enumerate}
        \item If $N_{k-1}^* = \emptyset$, then write $N_1=\emptyset$.
        \item If $N_{k-1}^* \neq \emptyset$,  delete the right most occurrence of each letter. Replace each remaining letter, $w_i$, with $\overline{(w_i-1)}$. Label this subword as $N_1$. 
        \item For each letter in $P_{k-1}^*$, delete the right most occurrence of each letter. Replace each remaining letter, $w_i$, with $w_i-1$. Label this subword as $P_1$.
        \item Draw a divider to the right of $P_1$, and label as $\pi_1 = N_1P_1$.
    \end{enumerate}
    \item Repeat this process until $N_1^*P_1^*$ has been relabeled as $\pi_{k-1} = N_{k-1}P_{k-1}$, omitting the part where the divider is added at the end.\label{enum:secondprocessend}
\end{enumerate}
This yields the partition \[\pi=\pi_0\,|\,\pi_1\,|\,\pi_2\,|\,\cdots\,|\,\pi_{k-1}, \]
which we now show is a type $B$ set partition.

Having begun with $w\in\ST_n$ and being precise with the construction of $\pi$, we ensure
that (in absolute value) every number $i\in[0,n]$ appears exactly once in $\pi$. 
In addition, we can ensure that $\pi_0$ is the $0$-block, since $1$ was contained in $P_{k}^*$, as $w$ is a flattened Stirling permutation.
Based on Notation \ref{theAdler}, 
this is sufficient to be a type $B$ set partition of $[-n,n]$ as desired.

Furthermore, $\varphi(\pi)$ is defined as the concatenation
\begin{equation}\varphi(\pi)=g(h(P_0))f(h(N_1))g(h(P_1))\cdots f(h(N_{k-1}))g(h(P_{k-1}))\label{eq:ugly1}
\end{equation}
where $g(h(P_0))=P_k^*$
and $f(h(N_i))g(h(P_i))=N_{k-i}^*P_{k-i}^*$ as described above. 
Hence, Equation \eqref{eq:ugly1} is precisely the description of $w$, 
which implies that $\varphi(\pi)=w$.
\end{proof}

The statement of Theorem \ref{thm:main} follows immediately from Theorem \ref{thm:bijection}.

In the proof of Theorem \ref{thm:bijection}, surjectivity was a complicated process. We present the following example to illustrate the steps in that part of the proof.

\begin{example}\label{ex:thmex}
Consider the following element of $\flatn(\ST_{11})$:
\vspace{-.02in}
\[w=1122993883(10)(10)(11)(11)46677455.\] 
To construct the pairs $N_{i}^*$ and $P_{i}^*$, we first label the subwords (from right to left) as in Steps \eqref{enum:firstProcess}-\eqref{enum:firstprocessend} in Theorem \ref{thm:bijection}:
   \[\underbrace{11}_{P_5^*}\underbrace{22}_{P_4^*}\underbrace{99}_{N_3^*}\underbrace{3883}_{P_3^*}\underbrace{(10)(10)(11)(11)}_{N_2^*}\underbrace{466774}_{P_2^*}\underbrace{55}_{P_1^*}.\] 
   Next, we create the pairs $N_{k-i}$ and $P_{k-i}$ for the set partition as in Steps \eqref{enum:secondprocess} - \eqref{enum:secondprocessend} in Theorem \ref{thm:bijection}:
    \[\pi = \underbrace{0}_{P_0}\,|\,\underbrace{1}_{P_1}\,|\,\underbrace{\overline{8}}_{N_2}\underbrace{2\,7}_{P_2}\,|\,\underbrace{\overline{9}\,\overline{(10)}}_{N_3}\underbrace{3\,5\,6}_{P_3}\,|\,\underbrace{4}_{P_4}.\]
We can verify that $\varphi(\pi)=w$, as desired.
\end{example}

Recall that the
Dowling numbers (OEIS \href{https://oeis.org/A007405}{A007405}) enumerate type $B$ set partitions \cite[Statement 6]{adler}.
Thus, Theorem~\ref{thm:main} readily implies the following.
\begin{corollary}\label{cor:dowling}
If $n\geq 1$, then 
$|\flatn(\ST_{n+1})|=D_n$,
where $D_n$ is the $n$th Dowling number.
\end{corollary}

In Subsection \ref{subsec:finalcount}, we provide a formula for the Dowling numbers (Corollary \ref{coro:dowling}). 

\section{Enumerating flattened Stirling permutations}\label{sec:prelim results}
In this section, we present enumerative results related to flattened Stirling permutations.

\subsection{Enumerating runs}
We begin by giving a formula for the number of runs in a flattened Stirling permutation and follow this with a formula for the maximum number of runs in a flattened Stirling permutation.

\begin{lemma}\label{lemma:run count}
Let $w \in \flatn({\ST_n})$ and $\pi=\pi_0\,|\,\pi_1\,|\,\cdots\,|\,\pi_k$ be its corresponding partition in $\Pi_{n-1}^B$ (as in Theorem \ref{thm:bijection}), with $\pi_i=N_iP_i$ (as in Definition \ref{def:NP}) for all $0\leq i\leq k$. Then 
\[ 1 + |\set{N_i: |N_i| \geq 1}| + |\set{P_i : |P_i| \geq 2}| \]
is the number of runs in $w$.
\end{lemma}

\begin{proof}
Let $w\in\flatn(\ST_n)$ and let $\mbox{runs}(w)$ denote the number of runs in $w$, where 
$\mbox{runs}(w)= \mbox{des}(w)+1$.
Let $\pi=\pi_0\,|\,\pi_1\,|\,\ldots\,|\,\pi_k$ be the type $B$ set partition corresponding to $w$ as given by Theorem \ref{thm:bijection}.
Recall that for all $0\leq i\leq k$, we write $\pi_i=N_iP_i$ where  $N_i$ and $P_i$ are as in Definition \ref{def:NP}.
We claim that for every index $i$ satisfying $|N_i| \geq 1$ and for every index $j$ satisfying $|P_j| \geq 2$ in $\pi$ encodes a distinct descent of $w$.\\

\noindent \textbf{Case 1}: If $P_i = \set{p_1}$ and $N_i = \emptyset$, then $f(h(N_i))g(h(P_i)) = p_1p_1$ creates no descents in $w$.\\

\noindent \textbf{Case 2}: Suppose $N_i = \emptyset$ and $P_i=\{p_1-1<p_2-1<\cdots<p_j-1\}$ where $|P_i|\geq 2$.
Then $g(h(P_i)) = p_1\, p_2\, p_2\, \cdots\, p_j\, p_j\, p_1$. 
By construction, $p_j\, p_1$ 
creates a descent since $p_1 < p_j$. 
Therefore, any $P_i$ in $\pi$ satisfying $|P_i|\geq 2$ contributes a descent to $w$, which in turn begins a new run.\\

\noindent \textbf{Case 3}: Suppose $N_i \neq \emptyset$.
Recall that by construction, we have that $P_i\neq \emptyset$, and the elements in $N_i$ satisfy $\min(h(N_i))>\min(h(P_i))$.
Then 
$f(h(N_i))g(h(P_i))$
has the property that the last value in $f(h(N_i))$ is $\max(h(N_i))$, and the first value in $g(h(P_i))$ is $\min(h(P_i))$. As $\max(h(N_i))>\min(h(P_i))$, we note this contributes a descent to $w$.
Therefore, any $N_i$ in $\pi$ satisfying $|N_i|\geq 1$ contributes a descent to $w$, which in turn begins a new run.\\

Thus, the number of runs in $w$ is given by 
\[\mbox{runs}(w) = \mbox{des}(w)+1 = 1 + |\set{N_i: |N_i| \geq 1}| + |\set{P_i : |P_i| \geq 2}|.\qedhere\]
\end{proof}

We now provide a formula for the maximum number of runs in a flattened Stirling permutation.

\begin{lemma}\label{lemma:max runs}
    The maximum number of runs in an element of $\flatn(\ST_n)$ is $\lceil \frac{2n}{3} \rceil$.
\end{lemma}

\begin{proof}
Let $w \in \flatn({\ST_n})$ and $\pi \in \Pi_{n-1}^B$ be its corresponding type $B$ set partition. 
By Lemma \ref{lemma:run count}, the total number of runs is given by
\[ 1 + |\set{P_i : |P_i| \geq 2}| + |\set{N_i : |N_i| \geq 1}|. \]
To maximize the number of runs, we must maximize the number of blocks, $\pi_i=N_iP_i$ with $|P_i| = 2$ and $|N_i|= 1$. 
Suppose $P_i=\{p_1-1<p_2-1\}$ and $N_i=\{n_1-1\}$. 
Then the subword $f(h(N_i))g(h(P_i))$ has the form 
\[n_1\,n_1\,p_1\, p_2\, p_2\,p_1\] with $n_1>p_1$ and $p_2>p_1$. 
Hence, this contributes
two descents, which are three letters apart in the subword. 
As $w$ consists of $2n$ letters, the maximum number of descents being contributed by blocks with $|N_i|=1$ and $|P_i|=2$ is given by $\frac{2n}{3}$. 
Starting at the beginning of the word, we can create a descent at every third letter. 
This creates no more than
$\frac{2n}{3}-1$ descents, where the ``$-1$'' accounts for the fact that the first possible index at which a descent may occur is index three. 

To maximize the number of descents, we would construct $\pi$ to have descents at indices $3k$ with $1\leq k\leq \lfloor\frac{2n-1}{3}\rfloor$, where the numerator $2n-1$ accounts for not being able to have a descent at the last index.
Since $\mbox{runs}(w)=\des(w)+1$, we have established that 
the maximum number of runs is given by 
$\lfloor\frac{2n-1}{3}\rfloor+1=\lceil\frac{2n}{3}\rceil$.
\end{proof}

We now illustrate Lemma~\ref{lemma:max runs} with the following example. 
\begin{example}
We consider $n=6,7,8$ as these integers are congruent to $0,1,2\mod 3$, respectively,  and account for every case of Lemma~\ref{lemma:max runs}. For each $n=6,7,8$, we construct a flattened Stirling permutation having a maximal number of runs.

\begin{itemize}
\item Let $n=6$ and note $\lceil\frac{2n}{3}\rceil=4$. To construct $w_1\in\flatn(\ST_6)$  with four runs, we must maximize the number of blocks, $\pi_i=N_iP_i$ with $|P_i| = 2$ and $|N_i|= 1$. One such option is given by 
$\pi =  13 \,|\, \overline{4}25 \,|\, 6.$
Then $
\varphi(\pi) = w_1 = 1\,3\,3\,1\,4\,4\,2\,5\,5\,2\,6\,6,
$
which has four runs.
    \item  Let $n=7$ and note $\lceil\frac{2n}{3}\rceil=5$. To construct $w_2\in\flatn(\ST_7)$ with five runs, we must maximize the number of blocks, $\pi_i=N_iP_i$ with $|P_i| = 2$ and $|N_i|= 1$. One such option is given by 
$\pi = 1\,3\,|\,\overline{4}\,2\, 5\,| \,\overline{7}\, 6.$
Then $
\varphi(\pi)=w_2 = 1\,3\,3\,1\,4\,4\,2\,5\,5\,2\,7\,7\,6\,6$,
which has five runs. 
    \item Let $n=8$ and note $\lceil\frac{2n}{3}\rceil=6$. To construct $w_3\in\flatn(\ST_8)$ with six runs, we must maximize the number of blocks, $\pi_i=N_iP_i$ with $|P_i| = 2$ and $|N_i|= 1$. 
One such option is given by 
$\pi=1\,3\,|\,\overline{4}\,2\,5\,| \, \overline{7} \,6\,8$.
Then $
\varphi(\pi)=w_3=1\,3\,3\,1\,4\,4\,2\,5\,5\,2\,7\,7\,6\,8\,8\,6
$,
which has six runs.

\end{itemize}

\end{example}

We remark that flattened Stirling permutations with a maximal number of runs are yet to be enumerated, and we pose this as an open problem in Section \ref{sec:future}.

\subsection{Flattened Stirling permutations with a small number of runs}
We now give formulas for the number of flattened Stirling permutations with one and two runs. 

\begin{proposition}\label{prop:2RunRec}
If $n\geq 1$, then 
\begin{enumerate}
    \item $|\flatn_1(\ST_n)|=1$, and \label{enum:2run1}
    \item $|\flatn_2(\ST_{n+1})|=2|\flatn_2(\ST_{n})|+2n-1$, 
with $|\flatn_2(\ST_1)|=0$.\label{enum:2run2}
\end{enumerate}
\end{proposition}
\begin{proof}
   For \eqref{enum:2run1}: We note that \[w=
    1\,1\,2\,2\,\cdots (n-1)\,(n-1)\,n\,n
    \]
is the only flattened Stirling permutation with one run.\\
    
\noindent For \eqref{enum:2run2}: We now construct the elements of $\flatn_2(\ST_{n+1})$ from those in  $\flatn_1(\ST_n)$ and $\flatn_2(\ST_{n})$. 
    {\setlength{\leftmargini}{.62in}
    \begin{enumerate}
    \item[\textbf{Case 1}:] The set $\flatn_1(\ST_n)$ has a
    unique flattened Stirling permutation with form $w=1122\cdots nn$. 
    There are two distinct ways of inserting $(n+1)(n+1)$ into $w$
    to create a flattened Stirling permutation of order $n+1$ with exactly two runs.
    \begin{enumerate}
        \item[(i)] 
        For all $j \in [n]$ insert $(n+1)(n+1)$ between the repeated value $j$ to build the Stirling permutation
        \[1\,1\,2\,2\cdots j\,(n+1)\,(n+1)\,j\cdots\,n\,n.\]
        The result of this insertion is a flattened Stirling permutation of order $n+1$ with two runs. 
        \item[(ii)]
        For all $j \in [n-1]$ insert the $(n+1)(n+1)$ after the repeated values $jj$ to build the Stirling permutation \[1\,1\,2\,2\cdots j\,j\,(n+1)\,(n+1)\,(j+1)\,(j+1)\cdots n\,n.\] 
        The result of this insertion is a flattened Stirling permutation of order $n+1$ with two runs. 
    \end{enumerate}
    This process exhausts all ways in which we can create flattened Stirling permutations with two runs 
    out of the flattened Stirling permutation with a single run. 
    If $n+1$ element is inserted anywhere else, it will create a Stirling permutation that is not flattened.
    Subcases (i) gives $n$ elements and (ii) yields $n-1$ elements of $\flatn_2(\ST_{n+1})$. Contributing a total of $2n-1$ elements in $\flatn_2(\ST_{n+1})$. \\
    
    \item[\textbf{Case 2}:] Take any element $w=w_1w_2\cdots w_{2n}\in \flatn_2(\ST_n)$. 
    By definition, there exists a unique index $i\in[2n-1]$ such that $w_i>w_{i+1}$. Inserting $(n+1)(n+1)$ immediately after $w_i$
    creates an element of $\flatn_2(\ST_{n+1})$. 
    Inserting $(n+1)(n+1)$
    immediately after the last value $w_{2n}$ of $w$ creates an element of $\flatn_2(\ST_{n+1})$. 
    This process exhausts all the ways in which we can create flattened Stirling permutations with two runs in $\flatn_2(\ST_{n+1})$ out of the flattened Stirling permutation with two runs in $\flatn_2(\ST_{n})$. 
    If you attempt to insert the $n+1$ element anywhere else, it will create a Stirling permutation that is not flattened.
    Hence, there are $2|\flatn_2(\ST_{n})|$ ways to create a new element in $\flatn_2(\ST_{n+1})$.
\end{enumerate}
}
Note that in Case 1, the constructed flattened Stirling permutations of order $n+1$ has the numbers $1$ through $n$ appearing in order from left to right. 
However, in Case~2, the constructed flattened Stirling permutations of order $n+1$ has unique values $i,j$ satisfying $1\leq i<j\leq n$ with at least one $j$ appearing to the left of at least one $i$. 
Thus, the set of elements in $\flatn_2(\ST_{n+1})$ created in Case 1 and Case 2 are disjoint.
Moreover, these are the only options to consider.
Thus, the total number of elements of $\flatn_2(\ST_{n+1})$ is given by
    \begin{equation*}
|\flatn_2(\ST_{n+1})|=2|\flatn_2(\ST_{n})|+2n-1,
    \end{equation*}
as desired.
\end{proof}

We now give a closed formula for $|\flatn_2(\ST_n)|$, which aligns with OEIS \href{https://oeis.org/A050488}{A050488}.

\begin{corollary}
    If $n\geq 1$, then $|\flatn_2(\ST_{n+1})|=3(2^n-1)-2n$.
\end{corollary}

\begin{proof}
    When $n=1$, $|\flatn_2(\ST_{2})|=1$ and $3(2^{1}-1)-2(1)=3-2=1$.
    Assume for induction that $|\flatn_2(\ST_{n})|=3(2^{n-1}-1)-2(n-1)$.
    Utilizing the recursion in Proposition \ref{prop:2RunRec} yields
    \begin{align*}
|\flatn_2(\ST_{n+1})|&=2|\flatn_2(\ST_{n})|+2n-1\\
&=2(3(2^{n-1}-1)-2(n-1))+2n-1\qquad\mbox{(by induction hypothesis)}\\
&=3(2^n-1)-2n,
    \end{align*}
    as desired.
\end{proof}

\subsection{Enumeration of \texorpdfstring{$|\flatn(\ST_n)|$}.}\label{subsec:finalcount}
We are now ready to provide a formula for the total number of flattened Stirling permutations of order $n$.

\begin{theorem} \label{them:compositions}
If  $n\geq 0$, then
        \[|\flatn (\ST_{n+1})|  =\displaystyle \sum_{i=0}^{n}\binom{n}{i}\sum_{k=0}^{n-i} 2^{n-i-k} S(n-i, k),\]
        where $S(a,b)$ denote the Stirling numbers of the second kind with $S(a,0)=0$ for all $a\geq 0$.
\end{theorem}
\begin{proof}
By Theorem \ref{thm:main}, we know $|\Pi_n^B|=|\flatn(\ST_{n+1})|$. Hence, we prove this result by constructing and thereby enumerating all $\pi$ which are type $B$ set partitions in $\Pi_{n}^B$. 

To begin we construct a zero-block. Suppose that there are $i+1$ elements in the zero-block. One of those elements is required to be zero. Thus, there are $\binom{n}{i}$ ways to construct the zero-block so that it has size $i+1$. We note that the zero-block always contains only positive elements, so, as $i$ ranges from $0$ to $n$, this construction yields all possible $\pi_0$.

We now construct all type $B$ set partitions with a fixed zero-block $\pi_0$ of size $i+1$. 
Suppose we have $k+1$ total blocks in $\pi$. We construct the remaining $k$ blocks, $\pi_1, \pi_2, \ldots , \pi_k$, out of the remaining $n-i$ elements (which are not in $\pi_0$). That is, first we must partition the set $[0,n]\setminus\pi_0$ of size $n-i$ into $k$ nonempty subsets. There are $S(n-i,k)$ ways to do this. 

Once we know what elements are in each block, then 
in each of the $k$ blocks, the minimal element must remain positive. However, the other $n-i-k$ remaining elements in the blocks $\pi_1,\pi_2,\ldots,\pi_k$ may be positive or negative.
There are $2^{n-i-k}$ ways in which to select the signs for those remaining values. 
Note we vary $k$ from $0$ to $n-i$, and thus there are $\sum_{k=0}^{n-i}2^{n-i-k} S(n-i, k)$ ways to build a type $B$ set partition for a fixed~$\pi_0$.

Finally, we must vary over $0\leq i\leq n$, which yields
\[
\displaystyle \sum_{i=0}^{n}\binom{n}{i}\sum_{k=0}^{n-i} 2^{n-i-k} S(n-i, k),
\]
as desired.
\end{proof}

Theorems~\ref{thm:main} and ~\ref{them:compositions} imply the following result.

\begin{corollary}\label{coro:dowling}
  If $n\geq 1$, then the  $n$th Dowling number is given by  \[D_n =\displaystyle \displaystyle \sum_{i=0}^{n}\binom{n}{i}\sum_{k=0}^{n-i} 2^{n-i-k} S(n-i, k).\]
\end{corollary}

\section{Open problems and future directions}\label{sec:future}

In our work, we were able to give a recursive formula for the total number of flattened Stirling permutations of order $n$ (Theorem~\ref{them:compositions}) and for the number of flattened Stirling permutations of order $n$ with exactly $k=2$ runs (Proposition \ref{prop:2RunRec}). When $k=3$, we have the following conjecture, which  was computationally verified for $1\leq n\leq 12$.
\begin{conjecture}   
 The number of Stirling permutations of order $n$ with exactly three runs is
    $|\flatn_3(\ST_n)| = 2\sum\limits_{k = 1}^{n-1} \binom{n-1}{k} \left( \sum\limits_{j = 2}^{n-1 - k} \binom{n-1 -k}{j} \right) 
    + \sum\limits_{k = 2}^{n-1} \binom{n-1}{k} \left( \sum\limits_{j = 2}^{n-1 - k} \binom{n-1 -k}{j} \right)
    + \sum\limits_{k = 3}^{n-1} (2^{k-1} - 2)\binom{n-1}{k}.$
\end{conjecture}

\begin{openproblem}\label{openprob:a}
    Give a formula for the number of flattened Stirling permutations of order $n$ with exactly $k$ runs.
\end{openproblem}

One small step toward solving Open Problem \ref{openprob:a} is to consider the case where $k=\lceil\frac{2n}{3}\rceil$ is the maximum number of runs. 

\begin{openproblem}
Let $\flatn_{\max}(\ST_n)$ denote the set of flattened Stirling permutations of order $n$ with a maximal number of runs. Give a formula for the cardinality of the set     $\flatn_{\max}(\ST_n)$, for all $n\geq 1$.
\end{openproblem}

Moreover, one could consider restricting to having runs of equal length. In which case we ask.

\begin{openproblem}\label{op2}
    Given $d$ a divisor of $n$, let $\flatn_{d,\,\mbox{equal}}(\ST_n)$ denote the set of flattened Stirling permutations with $d$ runs all of equal length $n/d$. What is the cardinality of the set $\flatn_{d,\mbox{equal}}(\ST_n)$ in terms of $n$ and $d$?
\end{openproblem}
A starting point for Open Problem  \ref{op2} is to consider the case $d=2$ where there are two runs of equal length.

Another way to generalize our study is to consider the set of  $m$-Stirling permutations of order $n$ on the multiset $\{1^m,2^m,\ldots,n^m\}$, which are permutations with each integer in $[n]$ appearing with multiplicity $m$, and the values between every repeated instance of $i$ are larger than $i$. 
We let $\ST_{n}^{m}$,  denote the set of $m$-Stirling permutations of order $n$. 
Table~\ref{tab:m-stirling} provides data on the number of elements in $\flatn(\ST_{n}^m)$ for $1\leq n\leq 7$ and $1\leq m\leq 5$. When possible, we give the OEIS sequence with which these values seem to agree. Based on that data, we give the following conjecture.

\begin{table}[h!]
    \centering
    \begin{tabular}{|c||c|c|c|c|}\hline
$n\setminus m$ & 2& 3& 4& 5\\\hline
\hline1&1&1&1&1\\\hline
2&2&3&4&5\\\hline
3&6&12&20&30\\\hline
4&24&63&128&225\\\hline
5&116&405&1008&2075\\\hline
6&648&3024&9280&22500\\\hline
7&4088&25515&96704&276875\\\hline
OEIS&\href{https://oeis.org/A007405}{A007405}&\href{https://oeis.org/A355164}{A355164}&\href{https://oeis.org/A355167}{A355167}&No OEIS hit\\\hline
    \end{tabular}
    \caption{Number of flattened $m$-Stirling permutations of order $n$.}
    \label{tab:m-stirling}
\end{table}

\begin{conjecture}\label{conj:mstirling} If $m\geq 2$ and $n\geq 1$, then
\[| \flatn (\ST_n^m)|= e^{-1/m} \displaystyle\sum_{k\geq 0} \frac{(mk + (m-1))^n }{ k!m^k },\] and satisfy the recurrence relation
\[| \flatn (\ST_n^m)| = (m-1) | \flatn (\ST_{n-1}^m)| + \sum_{k=1}^{n} \binom{n-1}{k-1}  m^{k-1}  | \flatn (\ST_{n-k}^m)|,\]
with $|\flatn (\ST_0^m)|=1$.
\end{conjecture}

Conjecture \ref{conj:mstirling} motivates our next open problem.
\begin{openproblem}
    Prove Conjecture \ref{conj:mstirling} in the affirmative or establish a correct formula for the number of elements in $\flatn(\ST_{n}^m)$. If possible, also give formulas for the number of elements in $\flatn(\ST_{n}^m)$ with exactly $k$ runs.
\end{openproblem}

\begin{remark}
    Since the submission of this article, Shankar  \cite{shankar2023enumeration} has provided a proof of Conjecture \ref{conj:mstirling} using generating functions. 
\end{remark}

\section*{Acknowledgements}
J.~Elder was partially supported through an AWM Mentoring Travel Grant. P.~E.~Harris was supported through a Karen Uhlenbeck EDGE Fellowship.

\bibliographystyle{plain}
\bibliography{bibliography}

\end{document}